\documentclass[12pt]{article}
\usepackage{amsmath}
\usepackage{amssymb}
\usepackage{amsthm}
\usepackage{fullpage}
\usepackage{mathrsfs}

\theoremstyle{plain}
\newtheorem{defin}{Definition}

\newtheorem{theorem}{Theorem}

\newtheorem{lemma}{Lemma}

\newcommand{\R}{\mathbb{R}}
\newcommand{\Z}{\mathbb{Z}}

\title{The Virtually Cyclic Classifying Space \linebreak of the Heisenberg Group}
\author{Andrew Manion, Lisa Pham, and Jonathan Poelhuis\thanks{The authors were funded by the NSF (DMS 03-54132) and the University of Notre Dame.  Address questions and comments to jpoelhui@nd.edu, amanion1@nd.edu, or lisa.pham@trincoll.edu. The results presented in this paper were first discovered in a more general setting by Wolfgang L{\"u}ck and Michael Weiermann; see \cite{Lueck2} for many results on the dimension of virtually cyclic classifying spaces. The authors of this paper discovered the particular result about the Heisenberg group independently during the summer of 2007. }}
\date{September 3, 2007}
\begin{document}
\maketitle

\begin{abstract}
We are interested in the relationship between the virtual cohomological dimension (or $vcd$) of a discrete group $\Gamma$ and the smallest possible dimension of a model for $B_{vc}\Gamma$.  Here, $B_{vc}\Gamma$ denotes the classifying space of $\Gamma$ relative to the family $vc$ of virtually cyclic subgroups of $\Gamma$.  In this paper we construct a model for $B_{vc}\Gamma_3$ of the Heisenberg group $\Gamma_3$.  This model has dimension $vcd(\Gamma_3)=3$.  We also prove that there exists no model of dimension less than 3. 
\end{abstract}

\section[Section 1(Introduction)]{Introduction. Statement of Results.}

Let $\Gamma$ be a discrete group.  A \emph{family} of subgroups of $\Gamma$ is a nonempty set $\mathscr F$ of subgroups of $\Gamma$ such that if $H \in \mathscr F$, then every subgroup and every conjugate of $H$ is also in $\mathscr F$.

\begin{defin}\emph{($\mathscr F$-universal $\Gamma$-space)} Let $\mathscr F$ be a family of subgroups of $\Gamma$.  Let $X$ be a $\Gamma$-CW-complex.  We say $X$ is an \emph{$\mathscr F$-universal $\Gamma$-space}, denoted $E_{\mathscr F}\Gamma$, if, for each subgroup $H \subset \Gamma$, \[X^H=\begin{cases}\text{contractible}&\text{if $H \in \mathscr F$}\\\emptyset& \text{if $H \notin \mathscr F$.}
\end{cases}\] 
\end{defin}
The orbit space $E_{\mathscr F}\Gamma/\Gamma$ is called the \emph{classifying space of $\Gamma$ relative to $\mathscr F$} and denoted $B_{\mathscr F}\Gamma$.  In recent years two families have appeared repeatedly and significantly in geometric topology.  These are $fin$ and $vc$.  Specifically, $fin$ denotes the family of finite subgroups of $\Gamma$, and $vc$ denotes the family of virtually cyclic subgroups of $\Gamma$.
(A group is virtually cyclic if it contains a cyclic subgroup of finite index.)

We are interested in the relationship between the virtual cohomological dimension (or $vcd$) of a discrete group $\Gamma$ and the smallest possible dimension of a model for $B_{vc}\Gamma$.  One says $vcd(\Gamma) \leq n$ if $\Gamma$ has a subgroup of finite index whose cohomological dimension $cd$ is less than or equal to $n$.  Recall $\Gamma$ has \emph{cohomological dimension} $\le n$ if every $\mathbb Z \Gamma$-module has a projective resolution of length less than or equal to $n$.  If $\Gamma$ is torsion-free, $cd(\Gamma)=vcd(\Gamma)$.

Work done by Eilenberg and Ganea first motivated interest in the dimension of the classifying space of a group $\Gamma$.   Eilenberg and Ganea \cite{Eilenberg} showed that for any group $\Gamma$ with $cd(\Gamma)=n$ there exists a model for $B\Gamma$ with dimension $n$ if $n \neq 2$.  Connolly and Kozniewski \nolinebreak\cite{Kozniewski} similarly showed that for a general class of discrete groups $\Gamma$ there exists a model for $B_{fin}\Gamma$ with dimension $n$ if $n \neq 2$.  Brady, Leary, and Nucinkis \cite{Brady}, however, showed that there exists a discrete group $\Gamma$ with $vcd(\Gamma)=n$ such that any model for $B_{fin}\Gamma$ must necessarily have dimension $n+1$ or greater.     

Our work focuses on the Heisenberg group, denoted $\Gamma_3$:   \[\Gamma_3=\left\{ \left.\left( \begin{array}{ccc}
1 & a & c \\
0 & 1 & b \\
0 & 0 & 1 \end{array} \right) \right\vert a,b,c \in \mathbb{Z} \right\}.\]  We are interested in the smallest dimension of a model for $B_{vc}\Gamma_3$.  Since $\Gamma_3$ is torsion-free, any virtually cyclic subgroup of $\Gamma_3$ is cyclic.    

Farrell and Jones \cite{Farrell} constructed a model for $B_{vc}\Gamma$ for any discrete subgroup $\Gamma$ of a linear group.  Their model has dimension $vcd(\Gamma)+1$.  Connolly, Fehrman, and Hartglass \cite{Hartglass} dealt with the case in which $\Gamma$ is a crystallographic group.  They showed that any model for $B_{vc}\Gamma$ must necessarily have dimension greater than or equal to $vcd(\Gamma)+1$, and they constructed a new model that realized this limit. 

Thus, for a discrete group $\Gamma$ such that $vcd(\Gamma)=n$, the results above suggest that a model for $B_{vc}\Gamma$ would need to have dimension $n+1$ or greater.

However, this is not the case.  Recall that $cd(\Gamma_3)=3$.  We prove  

\begin{theorem}\label{Main} There exists a model for $B_{vc}\Gamma_3$ of dimension 3.  There does not exist a model for $B_{vc}\Gamma_3$ of dimension less than 3.
\end{theorem}

\section[Section 2]{Basic Notions.}

Let $\Gamma$ be a discrete group.  Recall that a \emph{$\Gamma$-CW-complex} is a CW-complex $X$ which is also a $\Gamma$-space in such a way that each $\gamma \in \Gamma$ acts cellularly and fixes pointwise each cell which it stabilizes.

Let $H \subset \Gamma$ be a subgroup.  Let $A$ be an $H$-space.  Recall that the \emph{$\Gamma$-space induced from $A$} is \[\Gamma \times _H A =(\Gamma \times A)/\sim,\] where $\sim$ is the following equivalence relation: for all $\gamma$ and $\gamma' \in \Gamma$, and for all $a$ and $a' \in A$, $(\gamma,a) \sim (\gamma',a')$ if there exists an $h \in H$ such that $\gamma'=\gamma h$ and $a'=h^{-1}a$.

The isotropy group, $\Gamma_p$, of a point $p=[\gamma,a]\in \Gamma \times_H A$ is \[\Gamma_p=\gamma H_a \gamma^{-1}.\]  Therefore if $K$ is a subgroup of $\Gamma$, we see that \[(\Gamma \times _H A)^K =\{[\gamma,a] \in \Gamma\times_H A \vert K \subset \gamma H_a \gamma^{-1}\}.\]  Note that this set is empty if $K$ is not conjugate to a subgroup of $H$.  
  
Let $C$ be a complete set of left coset representatives of $H$ in $\Gamma$.  The quotient map $\pi: \Gamma \times A \to \Gamma \times_H A$ restricts to a homeomorphism \[C \times A \to \Gamma \times_H A.\]  From this we see that \[(\Gamma \times_H A)^K\cong\{(\gamma_i,a) \in C \times A \vert K \subset \gamma_i H_a \gamma_i^{-1}\}.\] 

Our model for the $vc$-universal $\Gamma_3$-space is a double mapping cylinder.  We denote the mapping cylinder of a continuous function $f: X \to Y$ as $Cyl(f)$.  If $g:X \to Z$ is also a continuous map, the double mapping cylinder $Cyl(f,g)$ is the quotient space obtained from $Cyl(f) \amalg Cyl(g)$ by identifying the copies of $X$ in each mapping cylinder with each other.

$Cyl(f)$ has $Y$ as a deformation retract.  However, if $f$ is a homotopy equivalence, then $X$ is also a deformation retract of $Cyl(f)$ (see Hatcher \cite{Hatcher}.)  It follows that $Cyl(f,g)$ has $Z$ as a deformation retract. 

\section[Section 3]{Construction of a Three Dimensional $vc$-universal $\Gamma_3$~-space.}

We denote the center of $\Gamma_3$ as $Z$.  As can be computed, \[Z=\left\{ \left.\left( \begin{array}{ccc}
1 & 0 & c \\
0 & 1 & 0 \\
0 & 0 & 1 \end{array} \right) \right\vert c \in \mathbb{Z} \right\}.\]
Let $N(H)=\{\gamma \in \Gamma_3 | \gamma H \gamma^{-1}=H\}$ be the normalizer of a subgroup $H \subset \Gamma_3$. 

Let $\mathscr Z=\{\text{$Z$ and its subgroups}\}$.  Let $A$ be the set of conjugacy classes of maximal cyclic subgroups of $\Gamma_3$ other than $Z$.  For each conjugacy class $\alpha \in A$, choose a representative $H_{\alpha}$.  A computation shows that $N(H_\alpha)=ZH_\alpha \cong Z \times H_\alpha$. Let $\mathscr F_\alpha= \{H_\alpha \,\text{and its subgroups}\}$. 

We will choose a universal $N(H_\alpha)$-space $\tilde{U_\alpha}$, a $\mathscr Z$-universal $N(H_\alpha)$-space $\tilde{V_\alpha}$, and an $\mathscr F_\alpha$~-universal $N(H_\alpha)$-space $\tilde{W_\alpha}$ for each $\alpha \in A$.  We will also choose a $\mathscr Z$-universal $\Gamma_3$-space $V$ such that $V \supset \tilde{V_\alpha}$ for every $\alpha \in A$.  We will define \[U_\alpha=\Gamma_3 \times_{N(H_\alpha)} \tilde{U_\alpha} \, \text{and} \, W_\alpha=\Gamma_3 \times_{N(H_\alpha)} \tilde{W_\alpha}.\] Finally, we will let \[U=\displaystyle\coprod_{\alpha \in A} U_\alpha\] and \[W = \displaystyle\coprod_{\alpha \in A} W_\alpha.\]

We will then choose $N(H_\alpha)$-maps \[f_\alpha:\tilde{U_\alpha} \to \tilde{V_\alpha}\] and \[g_\alpha:\tilde{U_\alpha} \to \tilde{W_\alpha}\] for each $\alpha \in A$. We will induce $\Gamma_3$-maps \[F_\alpha:U_\alpha \to V\] and \[G_\alpha:U_\alpha \to W_\alpha\] from $\iota_\alpha \circ f_\alpha$ and $g_\alpha$ respectively, where $\iota_\alpha$ is the inclusion $\tilde{V_\alpha} \hookrightarrow V$. Finally, we will define \[f=\displaystyle\bigcup_{\alpha \in A}F_\alpha:U \to V\] and \[g=\displaystyle\coprod_{\alpha \in A}G_\alpha:U \to W.\]  We will then define a double mapping cylinder \[E=Cyl(f,g).\] The space $E$ will be a three-dimensional $vc$-universal $\Gamma_3$-space.

It remains to choose the spaces $\tilde{U}_\alpha$, $\tilde{V}_\alpha$, $\tilde{W}_\alpha$, and $V$, as well as the maps $f_\alpha$ and $g_\alpha$, for each $\alpha \in A$.  Let $V$ be $\R^2$ with the following action: \[\gamma \cdot (x,y)=(x+a,y+b) \,\text{for}\, \gamma=\left[\begin{smallmatrix}1&a&c\\ 0&1&b \\ 0&0&1 \end{smallmatrix}\right] \in \Gamma_3.\]  Note that $\R^2$ with this action is indeed $\mathscr{Z}$-universal. Now, for each $\alpha$, $H_\alpha$ stabilizes precisely one line through the origin in $V$, and thus $N(H_\alpha) = ZH_\alpha$ stabilizes the same line. Define $\tilde{V}_\alpha$ to be this line (i.e. a copy of $\mathbb R$) with the $N(H_\alpha)$ action restricted from the $\Gamma_3$ action on $V$. Also, define $\tilde{W}_\alpha$ to be $\R$; let $N(H_\alpha) / H_\alpha \cong \Z$ act on $\tilde{W}_\alpha$ by translations. Finally, let $\tilde{U_\alpha}$ be $\tilde{V}_\alpha \times \tilde{W}_\alpha$ with the diagonal action. These spaces are all universal relative to the required families.

For each $\alpha \in A$, the projection homomorphisms of $N(H_\alpha) \cong Z \times H_\alpha$ onto its coordinates induce $N(H_\alpha)$-maps $f_\alpha:\tilde{U}_\alpha \to \tilde{V}_\alpha$ and $g_\alpha:\tilde{U}_\alpha \to \tilde{W}_\alpha$ respectively, since each $\tilde{V}_\alpha$ is a universal $N(H_\alpha) / Z$-space and each $\tilde{W}_\alpha$ is a universal $N(H_\alpha) / H_\alpha$-space.

\section[Section 4]{Proof that $E$ is $vc$-universal.}

\begin{lemma} $E=Cyl(f,g)$ above is $vc$-universal.
\end{lemma} 

\begin{proof}  To prove that $Cyl(f,g)$ is a $vc$-universal space for the Heisenberg group, we first calculate the isotropy group of each point $p \in E$.  We then use these groups $(\Gamma_3)_p$ to find the fixed set of each subgroup $K \subset \Gamma_3$.\\

Case 1: Let $p=[\gamma,a] \in W_\alpha$ for some $\alpha \in A$.  From \S 2, we know \linebreak $(\Gamma_3)_p=\gamma (N(H_\alpha))_a \gamma^{-1}=\gamma H_\alpha \gamma^{-1}$. \\ 

Case 2: Let $p\in V$.  Note that $V$ is a universal $\Gamma_3/Z$-space.  Thus, we conclude that $(\Gamma_3)_p=Z$.\\

Case 3: Let $p \in E \smallsetminus (V \cup W)$.  Then $(\Gamma_3)_p=\{1\}$.\\

We now find the fixed set $E^K$ for each subgroup $K \subset \Gamma_3$.  Recall that any virtually cyclic subgroup of $\Gamma_3$ is cyclic.\\

Case A: Suppose $K \in vc$ and $K \not\subset Z$.  We know that $K$ must be contained in exactly one maximal cyclic subgroup $\gamma H_\alpha \gamma^{-1}$, where $\gamma \in \Gamma_3$ and $\alpha \in A$.  By definition, \linebreak $E^K=\{p \in Cyl(f,g) \vert K \subset (\Gamma_3)_p\}$.  Since $K \subset \gamma H_\alpha \gamma^{-1}$, we can conclude from the above discussion that \[E^K=W_{\alpha}^{\phantom{\alpha}K}.\]  Now let $C_\alpha$ be a complete set of coset representatives of $N(H_\alpha)$ in $\Gamma_3$.  From \S 2 we know that \[E^K \cong \{(\gamma_i,a) \in C_\alpha \times \tilde{W}_\alpha \vert K \subset \gamma_i H_\alpha \gamma_i^{-1}\}.\]  But $K\subset \gamma_i H_\alpha \gamma_i^{-1}$ iff $\gamma H_\alpha \gamma^{-1} \subset \gamma_i H_\alpha \gamma_i^{-1}$ iff $\gamma_i \in \gamma N(H_\alpha)$.  Therefore if $\gamma_i \in \gamma N(H_\alpha)$, then $E^K \cong \{\gamma_i\} \times \tilde{W}_\alpha$, which is contractible.\\      

Case B: Suppose $K \in vc \smallsetminus \{1\}$ and $K \subset Z$.  Then $E^K=\{p \in V \vert K \subset (\Gamma_3)_p=Z\}=V$, which is contractible.\\

Case C: Suppose $K=\{1\}$.  Then $E^K=E$.  We must show that $E=Cyl(f,g)$ is contractible. Since $\tilde{U}_\alpha$ and $\tilde{W}_\alpha$ are both contractible, $g_\alpha$ is a homotopy equivalence for every $\alpha \in A$.  Therefore, so is $G_\alpha$ and thus so is $g$.  Hence, as stated in \S 2, $Cyl(f,g)$ deformation retracts to $V$, which is contractible.\\

Case D: Lastly, suppose $K \notin vc$.  Then $K$ is not cyclic, which implies $K \not\subset (\Gamma_3)_p$ for any $p \in E$, so $E^K=\emptyset$.\\

Thus, we have proven that $E=Cyl(f,g)$ is a $vc$-universal $\Gamma_3$-space.                     
\end{proof} 

\section[Section 5]{Proof that $dim(B_{vc}\Gamma_3) \ge 3$.}

\begin{lemma} There does not exist any model for $B_{vc}\Gamma_3$ of dimension less than three.
\end{lemma}
\begin{proof}
Let $B$ be the classifying space $E / \Gamma_3$. We will compute the three-dimensional homology group of $B$ to be nontrivial. Since all possible models for $B_{vc} \Gamma_3$ are homotopy equivalent, any model for $B_{vc} \Gamma_3$ must be a complex of dimension 3 or greater.

We will use $[\phi]$ to denote a map of orbit spaces induced from an equivariant map $\phi$.  The classifying space $B$ is homeomorphic to $Cyl([f],[g])$, which contains $Cyl([F_{\alpha}],[G_{\alpha}])$ as a subcomplex.

The orbit spaces $U_\alpha /\Gamma_3$ and $W_\alpha /\Gamma_3$ are homeomorphic to $\tilde{U}_\alpha /N(H_\alpha)$ and $\tilde{W}_\alpha /N(H_\alpha)$ respectively.  Hence, $Cyl([F_{\alpha}],[G_{\alpha}]) \cong Cyl([\iota_\alpha \circ f_{\alpha}],[g_{\alpha}])$, which in turn contains $Cyl([f_\alpha],[g_\alpha])$ as a subcomplex.

Since $f_\alpha$ and $g_\alpha$ are induced from the first and second coordinate projection homomorphisms of $Z \times H_\alpha$ respectively, the maps $[f_\alpha]$ and $[g_\alpha]$ are the first and second coordinate projections of $S^1 \times S^1$.  Thus $Cyl([f_\alpha],[g_\alpha])\equiv S^1 \ast S^1 \equiv S^3$.  Therefore $B$ contains $S^3$ as a subcomplex. Since $dim(E)=3$, we conclude $H_3(B) \neq 0$.
\end{proof}

\section[Section 6 (Conclusion)]{Conclusion.}
In this paper, we have constructed models for $E_{vc}\Gamma_3$, and thus $B_{vc}\Gamma_3$, of dimension $vcd(\Gamma_3)=3$.  It is known that for other groups $\Gamma$, such as crystallographic groups, $B_{vc}\Gamma$ must have dimension greater than $vcd(\Gamma)$.  It would be interesting to uncover which properties of $\Gamma_3$ on which these results depend.

\section[Section 7 (Acknowledgements)]{Acknowledgements.}  
We would like to thank Frank Connolly for motivation and insights.  Also, we thank all those at the University of Notre Dame who have facilitated our research, including Margaret Doig. 
\pagebreak

\bibliographystyle{plain}
\bibliography{myrefs}
\end{document}